\documentclass[12pt]{article}
\usepackage[a4paper, total={6in, 9in}]{geometry}

\usepackage{amssymb}

\usepackage{amsmath, amsfonts}
\usepackage{amsthm}
\usepackage{authblk}

\usepackage{multirow}
\usepackage{makecell}
\usepackage{array}
\usepackage{algorithm}%
\usepackage{algorithmicx}%
\usepackage{tabularx}
\usepackage{caption}
\usepackage{graphicx}

\newcolumntype{P}[1]{>{\centering\arraybackslash}p{#1}}

\def\R{\mathrm{\mathbb{R}}}
\def\F{\mathrm{\mathcal{F}}}
\def\N{\mathrm{\mathcal{N}}}
\def\M{\mathrm{\mathcal{M}}}
\def\L{\mathrm{\mathcal{L}}}
\def\C{\mathrm{\mathcal{C}}}

\def\u{\mathit{\hat{u}}}
\def\z{\mathit{\hat{z}}}
\def\x{\mathit{\hat{x}}}
\def\v{\mathit{\hat{v}}}
\def\w{\mathit{\hat{w}}}

\newtheorem{theorem}{Theorem}%

\newtheorem{proposition}[theorem]{Proposition}%

\newtheorem{corollary}[theorem]{Corollary}

\title{An inexact algorithm for stochastic variational  inequalities}

\author[1,3]{Emelin L. Buscaglia}
\author[2,4]{Pablo A. Lotito}
\author[1,3]{Lisandro A. Parente}

\affil[1]{Universidad Nacional de Rosario, Facultad de Ciencias Exactas, Ingenieria y Agrimensura, Av. Pellegrini 250 (2000), Rosario, Argentina}
\affil[2]{Universidad Nacional del Centro de la Provincia de Buenos Aires, Facultad de Ciencias Exactas, PLADEMA, Pinto 399 (7000), Tandil, Argentina}
\affil[3]{Centro Internacional Franco-Argentino de Ciencias  de la Informacion y Sistemas, CONICET - Universidad Nacional de Rosario, Bv. 27 de Febrero 210 bis, Rosario (2000), Argentina}
\affil[4]{CONICET, Buenos Aires, Argentina}
\affil[*]{Corresponding author: Emelin L. Buscaglia - buscaglia@cifasis-conicet.gov.ar }

\date{}                     
\begin{document}

\maketitle

\begin{abstract}
We present a new Progressive Hedging Algorithm to solve Stochastic Variational Inequalities in the formulation introduced by Rockafellar and Wets in 2017, allowing the generated subproblems to be approximately solved with an implementable tolerance condition. Our scheme is based on Inexact Proximal Point methods and generalizes the exact algorithm developed by Rockafellar and Sun in 2019, providing stronger convergence results.  We also show some numerical experiments in two-stage Nash games.

\end{abstract}


\section{Introduction}\label{intro}

The study of variational inequalities (VI) enables tackling nonlinear optimization problems as well as complementarity and equilibrium problems in a unified manner, providing a computational framework that leads to the development of many numerical schemes. 

Given a nonempty closed convex set   $C$ in a Hilbert space $H$ and a map $F:H\to H$, the VI consists in finding $x\in C$ such that  
\begin{equation}\label{IV2}
0\in F(x) + N_C(x),
\end{equation}
where $N_C(\cdot)$ is the normal cone respect to $C$. For a comprehensive analysis and a broad set of applications in the deterministic framework, we refer to \cite{fp:03}. In the stochastic setting, the objective function $F$ and the constraints set $C$ in \eqref{IV2} may depend on elements $\xi$ of a probability space $\Xi$, i.e., $F:H\times\Xi\to H$ and $C:\Xi\rightrightarrows H$ may be random (set-valued) maps. If the dependence involves only the objective function, the \textit{expected value} formulation (SV) just replaces the function $F$ in \eqref{IV2} by its expectation $\mathbb E_\xi(F(x,\xi))$ (see  \cite{GYR:99,IJT:18,IJOT:19}. When the uncertainty also concerns the constraints set, the \textit{expected residual minimization} approach (ERM) considers a collection of separated variational inequalities
and finds a point $x\in H$ that better approximates all the solutions simultaneously by minimizing the expectation of some residual (see \cite{cf:05,cwz:12,cpw:17}).

By noting that SV and ERM approaches do not capture the dynamic features of multistage stochastic problems, a new formulation for stochastic variational inequalities (SVI) was introduced in \cite{RW:17}, extending the SV formulation to the multistage setting. Based on the idea of \textit{nonanticipativity}  (from \cite{RW:76}), this approach provides optimality conditions for multistage stochastic programming (MSP) and gives a  dual structure involving nonanticipativity multipliers that allows performing stochastic decomposition and enables the numerical scheme of \cite{RS:18}. It relies on the \textit{progressive hedging algorithm} (PHA) \cite{RW:91} and convergence results are obtained by proving that it falls within a \textit{proximal point method} (PPM) \cite{mar:72,roc:76b}. 
 The PPM solves inclusions of the form $0 \in T(x)$, where $T:H\rightrightarrows H$ is a  maximal monotone operator (see \cite{RW:varan}), by solving at each iteration the regularized subproblems 
\begin{equation}\label{RMI}0\in c_kT(z)+z-z_k,\end{equation}
being $z_k$ the current approximation and $c_k$ a regularization parameter. Under mild conditions, the method converges to a solution, even if the subproblems \eqref{RMI} are solved inexactly, provided that the sequence of errors is summable. 

Although  the use of the PPM implicitly allows the possibility of inexact resolution of the subproblems in the PHA, the scheme in \cite{RS:18} is presented as an \textit{exact} type algorithm with  no discussion on how to handle approximate solutions. This is an essential issue in practice since it could be a difficult task to assure the summability conditions in numerical implementations and, even so, the computational burden can be quite expensive. 
Inexact PPMs with summability type convergence conditions have been developed \cite{bq:98,com:97,e:98} until constructive relative error criteria were introduced in \cite{solsva:hybrid,solsva:Teps} and unified in \cite{solsva:unif} for \textit{hybrid} type algorithms.  An extension was introduced in \cite{pls:08} by allowing the use of a variable metric in the subproblems, also providing tighter convergence results. 

The aim of the present contribution is to extend the PHA of \cite{RS:18} to allow inexact solutions of the subproblems with constructive tolerance conditions, following \cite{solsva:unif} and using the convergence results from \cite{pls:08}.

\section{Preliminaries}\label{prelim}

This chapter provides the reader with concepts, results, and notation that will be needed throughout this work.

 Consider the dynamic  $N$-stage model from \cite{RW:17}. At  stage $k$, a decision $x_k\in \R^{n_k}$ is made and then the information $\xi_k$ from a finite probability space $\Xi^k$, is revealed. The vector $\xi=(\xi_1,...,\xi_N)\in \Xi$ is called the observation vector or \textit{scenario} and $x=(x_1,...,x_N)\in \R^n$, with $n=n_1+...+n_N$, the \textit{decision vector}. For the aim of this work, it will be enough to establish $\Xi$ as a finite probability space, gathering the information of every $\Xi^k$, endowed with a probability function $p>0$. 

For multistage models, it is fundamental to assume the principle of  \textit{nonanticipativity}, that is, the decision $x_k$ is allowed to depend on the previous observations $\xi_1,..,\xi_{k-1}$ but not on those made after them. 
The key idea in \cite{RW:17} is to express the nonanticipativity as an orthogonality condition on specially designed Hilbert spaces. Let $\mathcal{L}_n$ be the space of decision functions 
$$\mathcal{L}_n= \{x(\cdot):\Xi \to \R^n \, \text{s.t.} \, x(\xi)= (x_1(\xi),\, ... \, x_N(\xi))\},$$ 
which is a Hilbert space with the expectational inner product given by
\begin{equation*}
    \langle x(\cdot), w(\cdot)\rangle=\sum\limits_{\xi\in \Xi} p(\xi) \sum\limits_{k=1}^N\langle x_k(\xi),w_k(\xi)\rangle.
\label{eq:piem}
\end{equation*} 

The \textit{nonanticipativity} subspace is defined as
\begin{equation*}
\mathcal{N}= \lbrace x(\cdot)  \in \mathcal{L}_n\mid  x_k(\xi) \mbox{ does not depend on } \xi_k, ..., \xi_N\rbrace.
\end{equation*}
In addition, scenario-dependent constraints can be considered for the decisions, that is, given nonempty closed convex sets $C(\xi)$,  define
\begin{equation*}
    \C= \lbrace x(\cdot) \in \mathcal{L}_n \mid \\  x(\xi) \in C(\xi) \mbox{ for all } \xi \in \Xi \rbrace.
\end{equation*}

Finally, from continuous functions $F(\cdot, \xi): \R^n \to \R^n$, $\xi\in\Xi$,  such that $F(x,\xi)= (F_1(x,\xi),...,F_N(x,\xi))$ with $F_k(x,\xi)\in \R^{n_k}$, it is defined an operator $\mathcal{F}: \mathcal{L}_n \rightarrow \mathcal{L}_n$ such that
\begin{equation*}
\mathcal{F}(x(\cdot)): \xi \mapsto F(x(\xi), \xi)=(F_1(x(\xi),\xi),...,F_N(x(\xi),\xi)).
\end{equation*}
Then, the \textit{Stochastic Variational Inequality in basic form associated to} $F(\cdot, \xi)$ \textit{and} $C(\xi)$ consists in finding $x(\cdot) \in \C \cap \N$ such that $$ -\F(x(\cdot)) \in N_{\C\cap \N}(x(\cdot)).$$ 

A dual structure is introduced by defining the subspace of nonanticipativity multipliers $\M$, the  orthogonal complement of $\N$,  expressed as
\begin{equation*}
\begin{split}
\mathcal{M}=\lbrace w(\cdot)& =(w_1(\cdot),\, ... \, ,w_N(\cdot))\in\mathcal{L}_n\mid \\ & \mathbb{E}_{\xi_k,...,\xi_N}[w_k(\xi_1,...,\xi_{k-1},\xi_k,...,\xi_N)]=0\rbrace,
\end{split}
\end{equation*}
where $\mathbb{E}_{\xi_k,...,\xi_N}$ denotes the conditional expectation knowing the initial components $\xi_1,...,\xi_{k-1}$.
The \textit{SVI in extensive form} is stated as the problem of finding $x(\cdot)\in \N$ and $w(\cdot)\in \M$ such that 
$$ -F(x(\xi),\xi) -w(\xi)\in N_{C(\xi)}(x(\xi)), \quad  \forall \xi \in \Xi.$$
A solution of the SVI in extensive form is a solution of the basic SVI and, under standard constraint qualifications ($ri(\C) \cap ri(\N) \neq \emptyset$, for instance), both problems are equivalent. 

The extensive formulation leads to the exact PHA for SVI introduced in \cite{RS:18}. At each iteration, from points $x_k\in\mathcal{N}$ and $w_k\in\mathcal{M}$, the scheme performs first a proximal step for the SVI in extensive form 
\begin{equation}\label{PHprox}
\begin{split}
 -F(x(\xi),\xi) -w_k(\xi)+ r(x_k(\xi)-x(\xi)) \in N_{C(\xi)} (x(\xi)),\\ 
\end{split}
\end{equation}
and  the new iterates are obtained through projections onto $\N$ and $\M$ as
\begin{equation}\label{PHproy}
    x_{k+1}(\cdot)=P_{\N}(x(\cdot)),\quad w_{k+1}(\cdot)=w_k(\cdot)+rP_{\M}(x(\cdot)).
\end{equation}

Considering the partial inverse of $\F+N_\C$ (see \cite{spi:83}), i.e., the set-valued operator  $T:H\rightrightarrows H$ implicitly given by
\begin{equation*}
    z \in T(y) \Leftrightarrow P_\mathcal{N}(z)+P_\mathcal{M}(y)\in (\mathcal{F}+N_\mathcal{C})(P_\mathcal{N}(y)+P_\mathcal{M}(z)),
\end{equation*}
the PHA turns to be a PPM for the operator $ATA$, where $A$ is the rescaling map $A=P_\mathcal{N}+rP_\mathcal{M}$ (see \cite[Theorem 1]{RS:18}), so the classical convergence theory applies, with linear convergence rate if $C(\xi)$ are polyhedral and the functions $F(\cdot,\xi)$ are affine. (See \cite[Theorem 2]{RS:18}).

Concerning inexact proximal schemes, the HIPPM of \cite{solsva:unif} solves approximately the subproblems \eqref{RMI} by computing a triplet
$(\z^k,\v^k,\varepsilon_k)\in H\times H\times \R_+$ such that
$$
\left\{\begin{array}{l} \v^k\in T^{\varepsilon_k}(\z^k) ,\\
c_k \v^k +\z^k-z_k =\delta^k 
\end{array}\right.\ \ \begin{array}{l}\text{and} \; 
\|\delta^k\|^2+2c_k\varepsilon_k\leq \sigma_k^2\left(\|c_k\v^k\|^2+
\|\z^k-z_k\|^2\right),\end{array}
$$
where $\sigma_k \in [0,1)$ is the error tolerance parameter,
and $T^\varepsilon$ is an outer approximation of $T$ that verifies $T^0=T$ (an $\varepsilon$-enlargement of  $T$, see \cite{bis:97}). 

If $z_k$ is not a solution, the next iterate is obtained by means of a projection onto an appropriate hyperplane, which is given by the explicit formula
$$z_{k+1}=z_k-\tau_k a_k \v^k,\text{ with } \quad a_k=\frac{\langle\v^k,z_k-\z^k \rangle-\varepsilon_k}{\|\v^k\|^2},$$
where $\tau_k\in(0,2)$ is a user choice parameter. If the solution set is not empty, the algorithm converges weakly to a solution (\cite[Theorem 7]{solsva:unif}), with linear convergence if $T^{-1}$ is Lipschitz-continuous at zero (\cite[Theorem 8]{solsva:unif}).

In the finite-dimensional context, an extension to the variable metric setting  was introduced in \cite{pls:08} considering the {\em generalized} proximal subproblems:
\begin{equation}\label{Mexact}
0\in c_k M_k T(z)+ z-z_k,
\end{equation}
given by a sequence $\{M_k\}$ of symmetric positive definite matrices (the framework in \cite{pls:08} is $\R^n$ but the results are valid on any finite-dimensional Hilbert space). If $M_k=I,$ for all $ k$, then the variable metric algorithm  falls within HIPPM, which is the scheme we will apply here. However, we are interested in the convergence rate result of \cite{pls:08}, which in the constant metric context stands as follows.

\begin{theorem}\label{rate:linear} Assume that $T^{-1}(0)\neq \emptyset$ and that $T^{-1}$ is \textit{outer Lipschitz-continuous} at $0$, i.e.,  $T^{-1}(0)$ is a closed set and there are constants $L_1 \geq 0 $ and $L_2\geq 0$ such that
\begin{equation}\label{def:outerLip}
    T^{-1}(v) \subset T^{-1}(0) + L_1 \|v\| B, \quad \forall v \in L_2 B,
\end{equation}
where $B=\{v\in \mathcal{H}\, \mid \, \|v\| \leq 1\}.$
Then the sequence $\{z_k\}$ generated by the HIPPM converges linearly to an element $z^*\in T^{-1}(0)$.
\end{theorem}
\begin{proof}
Apply \cite[Theorem 4.4]{pls:08} with $M_k=I, \forall k$.
\end{proof}
Note that condition \eqref{def:outerLip} does not force the solution set to be a singleton. 
In the next section, some advantages of considering this weaker condition (instead of Lipschitz continuity) are shown through simple examples.

\section{Inexact Progressive Hedging for SVIs}\label{sec:IPHA}

As we mentioned before, the main difficulty when applying the PHA lies in the exact resolution of the generated subproblems. To sort out this obstacle, we present an inexact version of the PHA. Let us consider monotone continuous functions $F(\cdot,\xi)$, closed convex sets $C(\xi)$ for each $\xi\in\Xi$ and  $r$ a positive parameter. For simplicity, we shall denote a function $x(\cdot)\in \mathcal{L}_n$ as $x$.

\bigskip

 \begin{algorithm}[H]
 \caption{Inexact Progressive Hedging Algorithm}\label{IPHA}
\begin{algorithmic}[1]
\item \textit{Inicialization:} Choose $x_0\in \mathcal{N}$, $w_0\in \mathcal{M}, \  \bar{\sigma} \in (0,1), \theta \in (0,1).$

\item \textit{Inexact Proximal Step:} Choose $\sigma_k\in [0,\bar{\sigma}).$ Find $\hat x^k$, $\hat w^k$ such that
\begin{equation}\label{ips} \left\lbrace \begin{array}{lll}
r(x_k(\xi)-\x^k(\xi))-w_k(\xi) \in (F+N_{C(\xi)}) (\w^k(\xi)), \quad \xi \in \Xi,\\ 
\delta^k = \w^k-\x^k\\
\end{array}
\right.\end{equation}and
\begin{equation}
\|\delta^k\| ^2 \leq  \sigma_k^2 \big( \|x_k-P_\N(\x^k)+ P_\M(\w^k)\|^2 + \|x_k-P_\N(\w^k)+P_\M (\x^k)\|^2\big).
\label{eq:conddelta}
\end{equation}

\item \textit{Actualization:} If $x_k=P_\N(\w^k) -P_\M (\x^k)$ STOP. Otherwise, choose  $\tau_k\in [1-\theta,1+\theta]$ and set
 \begin{equation*} 
 \begin{array}{lcl}
 x_{k+1}&=&x_k-\tau_k \alpha_k [x_k-P_\N(\x^k) ],
\\
 w_{k+1}&=&w_k+\tau_k \alpha_k rP_\M(\w^k), 
 \end{array}
 \end{equation*}
where 
\begin{equation*}
    \alpha_k= \frac{\langle x_k-P_\N(\x^k)+P_\M(\w^k),x_k-P_\N(\w^k)+P_\M(\x^k)\rangle}{\|x_k-P_\N(\x^k)+P_\M(\w^k)\|^2}.
\end{equation*}
Set $k:=k+1$ and go to the Inexact Proximal Step.
\end{algorithmic}
\end{algorithm}

Like the PHA, this algorithm looks for an approximate solution of the SVI in extensive form, depending on the initial point. Observe that if the stop condition is satisfied, then the inequality \eqref{eq:conddelta} becomes 
$$ \|\delta^k\|^2 \leq \sigma_k^2\|x_k - P_\N (\x^k) + P_\M (\w^k)\|^2.$$ 
Since $x_k\in \N$, the stop condition also implies $P_\M(\x^k)=0$ and $x_k=P_\N(\w^k)$, so $x_k - P_\N (\x^k) + P_\M (\w^k) = \w^k-\x^k=\delta^k$ and in consequence
$\delta^k=0$. Then, $\w^k=\x^k\in \N$ and 
$$x_k-\x^k = P_\N(\w^k)- \w^k=0.$$
Thus, the pair  $(x_k,w_k)$ solves the SVI in extensive form, and therefore, $x_k$ is a solution of the basic SVI. Analogous reasoning holds if we take as a stop condition the equality $x_k=P_\N (\x^k) + P_\M (\w^k)$.

Also, note that if each subproblem \eqref{ips} is solved exactly, i.e., if $\delta^k\equiv 0, \forall k$, then $\w^k=\x^k$ and the IPHA falls within the PHA provided that $\tau_k\alpha_k=1$. The following proposition says that it is possible to choose $\tau_k$ to achieve this equality.

\begin{proposition}

If the inequality in \eqref{eq:conddelta} is strengthened to 
\begin{equation}\label{prop:stronger}
    \|\delta^k\|\leq
\sigma_k \|x_k-P_\N(\w^k)+P_\M (\x^k)\|
\end{equation} and we choose
$\sigma_k\leq\theta$, then there exists
$\tau_k\in(1-\sigma_k,1+\sigma_k)$ such that $\tau_k \alpha_k=1$.

\end{proposition}
\begin{proof}
Assume that the pair $(x_k, w_k)$ is not a solution, so $x_k-P_\N(\w^k)+P_\M (\x^k)\neq 0$ and  $x_k-P_\N(\x^k)+P_\M (\w^k)\neq 0$. From \eqref{prop:stronger}, the definition of $\delta^k$ and the triangle inequality, we obtain
$$\|x_k-P_\N(\w^k)+P_\M (\x^k)\|- \|x_k-P_\N(\x^k)+P_\M (\w^k)\|\leq \sigma_k \|x_k-P_\N(\w^k)+P_\M (\x^k)\|$$and
$$   \|x_k-P_\N(\x^k)+P_\M (\w^k)\|- \|x_k-P_\N(\w^k)+P_\M (\x^k)\|\leq \sigma_k \|x_k-P_\N(\w^k)+P_\M (\x^k)\|,$$
so
\begin{equation}\label{prop:desig1}
(1-\sigma_k)\frac{ \|x_k-P_\N(\w^k)+P_\M (\x^k)\|}{ \|x_k-P_\N(\x^k)+P_\M (\w^k)\|}\leq 1\leq (1+\sigma_k)\frac{ \|x_k-P_\N(\w^k)+P_\M (\x^k)\|}{ \|x_k-P_\N(\x^k)+P_\M (\w^k)\|}
\end{equation}
From the definition of $\alpha_k$ and the Cauchy-Schwarz inequality, it holds that
\begin{equation}
\alpha_k\leq \frac{ \|x_k-P_\N(\w^k)+P_\M (\x^k)\|}{ \|x_k-P_\N(\x^k)+P_\M (\w^k)\|}.\label{prop:desig2}
\end{equation}
Also, using \eqref{prop:stronger} and \eqref{prop:desig1} we obtain
\begin{eqnarray}
\alpha_k&=&\frac{\|x_k-P_\N(\x^k)+P_\M (\w^k)\|^2+\|x_k-P_\N(\w^k)+P_\M (\x^k)\|^2-\|\delta^k\|^2}{2\|x_k-P_\N(\x^k)+P_\M (\w^k)\|^2}\nonumber\\
&\geq&\frac{1}{2}+\frac{(1-\sigma_k^2)\|x_k-P_\N(\w^k)+P_\M (\x^k)\|^2}{2\|x_k-P_\N(\x^k)+P_\M (\w^k)\|^2}\nonumber\\&\geq&\frac{1}{2}+\frac{1-\sigma_k^2}{2(1+\sigma_k)^2}=\frac{1}{1+\sigma_k}.\label{prop:desig3}
\end{eqnarray}
Finally, inequalities \eqref{prop:desig1},  \eqref{prop:desig2} and  \eqref{prop:desig3} 
 give
$$(1-\sigma_k)\alpha_k\leq 1 \leq (1+\sigma_k)\alpha_k,$$
which proves the claim.
\end{proof}

The former proposition establishes that if we can solve the subproblems in Algorithm \ref{IPHA} satisfying \eqref{prop:stronger} (which is obviously the case when they are solved exactly), then the new iterates can be obtained by the simple formula
$$ \begin{array}{lcl}
 x_{k+1}&=&P_\N(\x^k),\\
 w_{k+1}&=&w_k+ rP_\M(\w^k),
\end{array}$$
which is the same actualization formula of the PHA of \cite{RS:18}, so IPHA actually extends that scheme.

Assuming that the stop condition never holds, we obtain a sequence that, we shall prove, converges to a solution. With this goal in mind, we shall proceed in the manner of Rockafellar and Sun, by using the HIPPM of \cite{solsva:unif}.

\begin{theorem}
The iterations given by the IPHA  are equivalent to those obtained by the HIPPM from \cite{solsva:unif} for the mapping $ ATA $ by setting $\varepsilon_k=0$ and $c_k= r^{-1}$ for all $k$, where
\begin{equation*}
    z \in T(y) \Longleftrightarrow P_\mathcal{N}(z)+P_\mathcal{M}(y)\in (\mathcal{F}+N_\mathcal{C})(P_\mathcal{N}(y)+P_\mathcal{M}(z))
\end{equation*}
and $A:\mathcal{L}_n \to \mathcal{L}_n$ a symmetric positive-definite and invertible linear mapping given by
\begin{equation*}
A (u)= P_\mathcal{N}(u)+rP_\mathcal{M}(u).
\end{equation*}
\label{theo:alg.equiv}
\end{theorem}
\begin{proof}
Let $x_k\in\N$ and $w_k\in\M$ be the current iterates of the IPHA, and define $z_k=x_k-r^{-1}w_k$.
Then $P_\N(z_k)=x_k$ and $P_\M(z_k)=-r^{-1}w_k$. Observe that if $0\in ATA(z_k)$ then  $0\in T(x_k-w_k)$ and the definition of $T$ implies that $(x_k,w_k)$ solves the SVI in extensive form, then the problem reduces to look for a zero of $ATA$.
Consider, from $z_k$, an iteration of the HIPPM for $ATA$ setting $\varepsilon_k=0$ and $c_k=r^{-1}.$ That is, choose $\sigma_k\in [0,\overline{\sigma}) $ and find  $\v^k,\z^k$ such that 
\begin{equation}\label{IPS1} \left\lbrace \begin{array}{lll}
\v^k \in ATA(\z^k)\\ 
\delta^k = r^{-1}\v^k +\z^k-z_k, \\
\end{array}
\right.\end{equation}
verifying
\begin{equation}
\|\delta^k\| ^2 \leq  \sigma_k^2  \big( \|r^{-1}\v^k\|^2 + \|\z^k- z_k\|^2\big).
 \label{eq:conddelprox}
\end{equation}
The algorithm stops if $z_k-\z^k=0$ (in which case $z_k$ is a zero of $ATA$ and thus $(x_k,w_k)$ solves the SVI), otherwise, select $\tau_k \in [1-\theta,1+\theta]$ and define 
\begin{equation}z_{k+1}= z_k-\tau_ka_k \v^k,\label{update1}
\end{equation}
where  
$$a_k= \frac{\langle \v^k,z_k-\z^k\rangle }{\|\v^k\|^2}. $$ 
From \eqref{IPS1}, we have 
\begin{equation*}
    r(\delta^k +z_k-\z^k)\in ATA(\z^k). 
\end{equation*}
Set $u_k=A(z_k)$ y $ \u^k=A(\z^k).$ Note that 
$$x_k=P_\N(z_k)=P_\N(u_k) \quad y \quad w_k=-rP_\M(z_k)=-P_\M(u_k).$$ In this setting, the inclusion in \eqref{IPS1} can be written as
$$    rA^{-1}\delta^k + rA^{-2}(u_k-\u^k) \in T(\u^k),
$$
which, by the definition of $T$ and noting that $A^{-1}=P_\N+r^{-1}P_\M$, gives
\begin{equation}\label{IPS2}
    r P_\N(\delta^k) + rP_\N(u_k-\u^k) +P_\M(\u^k) \in (\F+N_\C)(P_\N(\u^k)+P_\M(\delta^k) + r^{-1} P_\M(u_k-\u^k)).
\end{equation}

Define $\w^k=P_\N(\u^k)+P_\M(\delta^k) + r^{-1}P_\M(u_k-\u^k)$. Hence,
\begin{eqnarray}\label{whatN}P_\N(\w^k)&=&P_\N(\u^k)=P_\N(\z^k),\\
\label{whatM}P_\M(\w^k)&=&P_\M(\delta^k) + r^{-1}P_\M(u_k-\u^k)=P_\M(\delta^k+z_k-\z^k),\end{eqnarray} 
and \begin{equation}\label{uhatM}P_\M(\u^k)=rP_\M(\delta^k-\w^k)-w_k.\end{equation}
Thereby, using these replacements in inclusion \eqref{IPS2}, it results
\begin{equation*}
    rP_\N(\delta^k) +rx_k-rP_\N(\w^k)+rP_\M(\delta^k)-rP_\M(\w^k)-w_k\in (\F+N_\C)(\w^k),
\end{equation*}
that is, 
\begin{equation*}
    r(x_k+\delta^k -\w^k)-w_k \in (\F+N_\C)(\w^k).
\end{equation*}
From the above inclusion, defining $\x^k=\w^k-\delta^k$, we obtain \eqref{ips}. Also, \eqref{whatN} and \eqref{whatM} give

\begin{eqnarray}
    \v^k&=&r(\delta^k+z_k-\z^k)=r\big[P_\N(z_k-(\z^k-\delta^k))+P_\M(\delta^k+z_k-\z^k)\big]\nonumber \\
     &=&r\big[x_k-P_\N(\x^k)+P_\M(\w^k))\big]\label{vhat1}
\end{eqnarray}
and 
\begin{equation}
    z_k-\z^k=P_\N(z_k-\z^k)+P_\M(z_k-\z^k)=x_k-P_\N(\w^k)+P_\M(\x^k)\label{zkzhat}
\end{equation}
so \eqref{eq:conddelprox} takes the form
\begin{equation*}
\begin{split}
    \|\delta^k\|^2 &\leq \sigma_k^2\big(\|r^{-1}\v^k\|^2 +\|z_k-\z^k\|^2\big)\\
    &=\sigma_k^2\big(\|x_k-P_\N(\x^k)+P_\M(\w^k))\|^2 +\|x_k-P_\N(\w^k)+P_\M(\x^k)\|^2\big),
\end{split}
\end{equation*}
which is just condition \eqref{eq:conddelta}. In addition, \eqref{zkzhat} implies that the stopping criterion $z_k-\z^k=0$ is equivalent to the one in IPHA, i.e.,  
$x_k-P_\N(\w^k)+P_\M(\x^k)=0$. If it is not satisfied, the new iterate $z_{k+1}$ given by \eqref{update1} can be written as $z_{k+1}=x_{k+1}-r^{-1}w_{k+1}$ so
\begin{equation*}
    \begin{split}
        x_{k+1}&=P_\N(z_{k+1})=P_\N(z_k-\tau_k a_kr(\delta^k+z_k-\z^k)=x_k-\tau_k a_k r(x_k-P_\N(\x^k))\\
        w_{k+1}&=-rP_\M(z_{k+1})=-rP_\M(z_k-\tau_k a_kr(\delta^k+z_k-\z^k))=w_k+\tau_k a_kr^2 P_\M(\w^k).
    \end{split}
\end{equation*}

From \eqref{vhat1} and \eqref{zkzhat}, $a_k$ can be expressed as
\begin{equation*}
    a_k= \frac{\langle x_k-P_\N(\x^k)+P_\M(\w^k),x_k-P_\N(\w^k)+P_\M(\x^k)\rangle}{r\|x_k-P_\N(\x^k)+P_\M(\w^k)\|^2}
\end{equation*}
so defining $\alpha_k=ra_k$, we achieve the same updates of IPHA.  
 
\end{proof}

Theorem \ref{theo:alg.equiv} allows us to apply the theory from \cite{solsva:unif} and \cite{pls:08} to establish the convergence properties of algorithm \ref{IPHA}. For this purpose, it is convenient to introduce an equivalent norm in $\L_n$ that properly associates the iterates $(x_k,w_k)$ of IPHA with the iterates $z_k$ of HIPPM. Specifically, we use the norm induced by the symmetric positive definite matrix $M_r= \left(\begin{array}{cc}
    I_\N &0  \\
    0 &r^{-2}I_\M 
\end{array}\right)$, i.e., the norm $\|\cdot\|_{M_r} $ given by
$$\|(x,w)\|^2_{M_r}=\langle (x,w),M_r(x,w)\rangle=\langle (x,w),(x,r^{-2}w)\rangle=\|x\|^2+r^{-2}\|w\|^2.$$

\begin{corollary}
If the SVI in extensive form has at least one solution, then the sequence $\lbrace( x_k,w_k)\rbrace$ generated by the IPHA converges to $(x^*,w^*)$, a solution of the SVI in extensive form, being $x^*$ a solution of the basic SVI. 
\label{coro:iphaconver}
\end{corollary}
\begin{proof}
Note that the existence of a solution of SVI in extensive form  is equivalent to $(ATA)^{-1}(0) \neq \emptyset$ and apply (\cite{solsva:unif}, Theorem 6).
\end{proof}

\begin{corollary}\label{coro:Lip1}

If  $(ATA)^{-1}$ is Lipschitz continuous at zero, then the convergence is linear in the norm induced by $M_r$.
\end{corollary}

\begin{proof} Apply \cite[Theorem  8]{solsva:unif}.
\end{proof}

\begin{corollary}
In addition to the assumptions of Corollary  \ref{coro:iphaconver}, if $(ATA)^{-1}$ is outer Lipschitz continuous at 0 then the sequence $\{(x_k,w_k)\}$ converges linearly to a solution $(x^*,w^*)$ in the norm induced by $M_r$. 
\label{coro:outerLip}
\end{corollary}

\begin{proof}
Note that, since $z_k=x_k-r^{-1}w_k$, it holds that  $\|z_k\|=\sqrt{\|x\|^2+r^{-2}\|w\|^2}=\|(x_k,w_x)\|_{M_r}$ and apply Theorem \ref{rate:linear}.

\end{proof}

The condition $(ATA)^{-1}$ Lipschitz continuous at zero is quite strong, there are simple functions for which $(ATA)^{-1}$ do not satisfy it for not being  single-valued at zero, consider, for example, $\Xi=\{0,1\}$, $C(\xi)= \mathbb R_+^{2}$ and $F(x,\xi)= M_\xi x + b_\xi$ where
$$M_\xi = \left(\begin{array}{cc}
2 & 1 \\
1+\xi & 2-\xi
\end{array} \right) \quad \text{ and } \quad 
b(\xi) = \left(\begin{array}{cc}
1 \\
1+\xi 
\end{array} \right). $$ 
The \textit{outer Lipschitz continuity} is weaker. In particular, if $F$ is a linear function, the mapping $F+N_C$ becomes a polyhedral mapping if $C$ is a closed convex polyhedral set, therefore $F+N_C$ is outer Lipschitz continuous (see \cite{DR:09}). It is not difficult to show that if $S$ is a polyhedral set-valued mapping, then its partial inverse also has this property, which constitutes the key to showing that for linear functions $(ATA)^{-1}$ is outer Lipschitz. So Corollary \ref{coro:outerLip} is a kind of extension of  \cite[Theorem 2]{RS:18}.

\section{Implementation issues and numerical examples}

The results in the previous section show that the IPHA has suitable convergence properties under rather mild conditions. Its performance will strongly depend on how efficiently the subproblems are solved. 

For instance, in the case of Lipschitz continuous functions, it is always possible to obtain approximated solutions of the subproblems via fixed-point algorithms (FPA) for adequate choices of the  parameter $r$. Indeed, provided Lipschitz continuous functions  $F(\cdot, \xi)$ and closed convex sets $C(\xi)$, the subproblems to be solved result 
$$ -F(x(\xi),\xi) +r(x_k(\xi)-x(\xi))-w_k(\xi) \in N_{C(\xi)} (x(\xi)), \quad \xi \in \Xi$$
which are equivalent to solve the following equations
$$x(\xi)=P_{C(\xi)}\Big(x_k(\xi)-\frac{1}{r}w_k(\xi)-\frac{1}{r}F(x(\xi),\xi))\Big), $$
where $P_{C(\xi)}$ is the projection mapping onto $C(\xi)$ for each $\xi \in \Xi$. Then, we are looking for fixed points of the operators
$$\Phi_{\xi}^k(z)=P_{C(\xi)}\Big(x_k(\xi)-\frac{1}{r}w_k(\xi)-\frac{1}{r}(F(z),\xi) )\Big).$$

To guarantee the convergence of the FPA, we require $F(\cdot, \xi)$ to be contractive. Then, for any $z, \bar{z}$ we have
\begin{equation}
\Big\|\Phi^k_{\xi}(z)-\Phi^k_{\xi}(\bar z)\Big\|\leq \frac{1}{r}\mu_{\xi}\|\bar z-z\|,
     \label{eq:Lip1}
\end{equation}
where $\mu_\xi$ is the Lipzchitz modulus of $F(\cdot, \xi)$. It is clear that the algorithm will converge when $r> \max\{\mu_\xi : \xi \in \Xi \}$.
\medskip

The FPA is effective but its linear convergence rate depends on the choice of large values for the parameter $r$, which can be counterproductive for the IPHA. Therefore, when using this kind of algorithm, there is a compromise in the choice of $r$. When the subproblems structure allows it, it could be preferable to implement other resolution strategies.

Another possible approach to approximate the subproblems solutions is based on the search of zeros of the mappings $\Phi_\xi^k - I$ instead of fixed points of $\Phi_\xi^k$. In the case that  $\Phi_\xi^k-I$ is semismooth, it is suitable to consider semismooth Newton methods (SNM) in the manner of  \cite{fp:03,IS:14}. When $F(\cdot, (\xi))$ are linear and $C(\xi)$ are polyhedral,  $\Phi_\xi^k-I$ turn to be piecewise linear. Therefore, they are  strongly semismooth and, under standard assumptions, SNM has a local quadratic convergence rate (see \cite{fp:03}), which is considerably better than the fixed point convergence. Also, there is no restriction for the choice of $r$.

 In the following sections, we shall illustrate these advantages using numerical examples for an equilibrium problem.

\subsection{Application example}
 We shall consider two companies E$_1$ and E$_2$ producing energy with given cost and price, each dependent on certain external factors (rain, fuel, availability of resources, etc.) and on the other company's production, with some restrictions regarding the demand and the production capacity, also dependent on these factors. The production is made in two stages; at stage one, the production level of E$_i$ is given by $x_i^1\in \mathbb{R}^{m_i}$ with a production cost $c_i^1x_i^1$, where $c_i^1$ is a positive constant in $\mathbb{R}^{m_i}$. We assume that the price at this stage, $p^1$, only depends on the production of both companies, namely, 
 $$p^1(x_1^1,x_2^1 )=\alpha^1 (a^1-\sum_{j}x_{1_j}^1-\sum_j x_{2_j}^1) $$
 with $\alpha^1$ and  $a^1$ positive constants. After that, the number of resources, for instance, may change, which would cause prices and costs to vary. We will represent these changes by random variables in a probability space associated with a set $\Xi=\{\xi_1,...,\xi_s\}\subset \mathbb{R}$ with a strictly positive probability distribution function. At second stage, if $\xi \in \Xi$ occurs, the cost for E$_i$ and the price are  given by 
$$  c_i^2(\xi)x_i^2 \quad \mbox{ and } \quad  p^2(x_1^2,x_2^2, \xi)=\alpha^2(\xi)(a^2(\xi)-\sum_{j}x_{1_j}^2-\sum_j x_{2_j}^2),$$
respectively, where $c_i^2(\cdot), \, \alpha^2(\cdot)$ and $a^2(\cdot)$ are positive. Therefore, for each E$_i$, the function to optimize is 
\begin{equation*}g_i(x,\xi)=c_i^1x_i^1-p^1(x^1_1,x^1_2)\sum_{j}x_{i_j}^1+c_i^2(\xi)x_i^2-p^2(x^2_1,x^2_2,\xi)\sum_{j}x_{i_j}^2.
\end{equation*} 
We can also consider some restrictions, depending on those random factors, about the production capacity  of each company, represented by the sets $$C_i(\xi)= \{(x^1_i,x^2_i): \ x^1_i,x^2_i \geq 0, \  x_{i_j}^1+x^2_{i_j}\leq\ell_{i_j}(\xi)\}.$$
Setting  $\Xi=\{\xi_1,...,\xi_s\}$, $n=2(m_1+m_2)$ and $F(\cdot, \xi):\R^n\to \R^n $ as
$$  F(x,\xi)=(\nabla_{x_1^1,x_1^2}g_1(x,\xi),\ 
\nabla_{x_2^1,x_2^2}g_2(x,\xi)),$$
finding a Nash equilibrium between the companies is to solve the SVI associated to $F$ and $C_i$, that is, to find $x(\cdot)\in \mathcal L_s $ such that $-\F (x(\cdot))\in N_{\C\cap \N} (x(\cdot))$ where
 $$ \F(x(\cdot))(\xi)=F(x(\xi),\xi),\quad
\mathcal{C}=\mathcal{C}_1\times \mathcal{C}_2,\quad
\N=\N_1\times \N_2,$$ 
with for each $i=1,2$ 
$$\mathcal{C}_i=\{ x_i(\cdot)=(x_i^1, x_i^2(\cdot)) : x_i(\xi) \in C(\xi)\} $$  and  $$ \mathcal{N}_i=\{x_i(\cdot)=(x_i^1, x_i^2(\cdot)) : x_i^1 \text{ is constant}\}.$$
 
  In the $k$th iteration of the IPHA, given $x_k $ and $w_k$, we need to approximately solve, for each $\xi \in \Xi,$ the inclusion 
$$ -F(x(\xi),\xi) +r(x_k(\xi)-x(\xi))-w_k(\xi) \in N_{C(\xi)} (x(\xi)).$$
  In this particular instance, reordering terms, the problems turn to be
\begin{equation}\label{subprob} \mbox{ find } x(\xi) \mbox{ such that } -M_{\xi} x(\xi) - b_\xi  -w_k(\xi)+ r[ x_k(\xi)-x(\xi)]\in N_{C(\xi)}(x(\xi)), \end{equation} 
  where $x(\xi)=(x_1^1,x_2^1,x_1^2(\xi),x_2^2(\xi))$  and $M_\xi \in \mathcal{M}(\mathbb{R},n)$

Observe that $F(\cdot,\xi)$ will be Lipschitz continuous with modulus $\mu_\xi=$   $\|M_\xi\|$. Therefore, as we established before, based on \eqref{eq:Lip1} the FPA solves the problem if $r> \max \{\|M_\xi\|: \xi \in \Xi\}$.

\subsection{Numerical experiments}
Here, we show some numerical trials for the application example described in the previous section. The algorithm performance was tested under different size problems, increasing the number of scenarios, the players' dimension, or the parameter $r$, using FPA or SNM alternatively.
We have coded the IPHA in Python 3.9, on a 1.80 GHz, 8 GB RAM, Intel Core i5 processor laptop. We set a fixed tolerance parameter $\sigma_k=0.5$ and the stopping condition was implemented as $\|x_k-P_\N(\w^k) -P_\M (\x^k)\|\leq 10^{-5}$. 
 The probability space and the information relating to prices, costs, and production limits are randomly generated.  
 
We show in table \ref{table:esce} the results from 50 to 500 scenarios. considering $m_{1}=m_{2}=10$, $r= \max \{\|M_\xi\|: \xi \in \Xi\}+ 0.1 $ for the FPA and $r=20$ for SNM. The number of iterations and the execution time increase considerably when using FPA while the first one remains almost the same and the second increases linearly for SNM. Actually, for 500 scenarios, solving the subproblems with the FPA requires more than 17 minutes, while the SNM takes less than 1 minute.
 
\begin{table}[H]
\captionsetup{width=.7\textwidth}
\setlength\extrarowheight{2pt}
    \centering
    \resizebox{0.6\textwidth}{!}{%
    \begin{tabular}{ P{25mm}|P{15mm} P{15mm}| P{15mm} P{15mm}}
    \hline
        \multirow{2}{*}{\textbf{Scenarios}} & \multicolumn{2}{c |}{ \textbf{Time avg. (s)}}  & \multicolumn{2}{c}{ \textbf{Iter. avg.}} \\ 
         & \textbf{FPA} & \textbf{SNM} & \textbf{FPA} & \textbf{SNM} \\ \hline
        50 & 11.80 & 4.35 & 2046 & 53 \\ \hline
        150 & 103.97 & 14.20 & 5703 & 51 \\ \hline
        300 & 410,55 & 29.81 & 12339 & 52 \\ \hline
        500 & 1013.45 & 65.2 & 18748 & 57 \\ 
            \hline
    \end{tabular}}
    \medskip
    
    \caption{Average time and iterations for each method when the number of scenarios increases. }
    \label{table:esce}
\end{table}

Table \ref{table:dimension} shows the evolution of the average time and the average number of iterations when executing the algorithm with 50 scenarios while the dimension of the problem ($m_1$ and $m_2$) increases. To clarify the size of the problems we are working with, observe that the solutions are functions $x:\Xi \to \mathbb{R}^{n}$ therefore, in this setting they lie between  $\mathbb{R}^{50\times 40}$ and $\mathbb{R}^{50\times 1000}$. The behavior is similar to the observed in table \ref{table:esce}.
 \begin{table}[H]
    \centering
\captionsetup{width=.7\textwidth}
\setlength\extrarowheight{2pt}
\resizebox{0.6\textwidth}{!}{%
    \begin{tabular}{ P{25mm}|P{13mm} P{13mm}|P{13mm} P{13mm}}
    \hline
          \multirow{2}{*}{\textbf{Dimension}} &
          \multicolumn{2}{ c| }{\textbf{Time avg. (s)}}   & \multicolumn{2}{ c }{ \textbf{Iter. Avg.}} \\ 
        ~ & \textbf{FPA} & \textbf{SNM} & \textbf{FPA} & \textbf{SNM} \\ \hline
        [10,10] & 11.80 & 3.90 & 2046 & 35 \\ \hline
        [50, 50] & 286.25 & 15.62 & 8713 & 27 \\ \hline
        [100, 100] & 918.87 & 29.37 & 16543 & 24 \\ \hline
        [250, 250] & 3292.73 & 218.88 & 28182 & 26 \\ \hline
    \end{tabular}}
    \caption{Average time and iterations for each method when the dimension increases, for 50 scenarios. }
    \label{table:dimension}
\end{table}

Lastly, we test the algorithm performance in terms of the choice of $r$. Table \ref{table:param_r} shows the average number of iterations and execution time with $r= \max \{\|M_\xi\|: \xi \in \Xi\}+ 0.1 +j$ when using the FPA and $r=j$ for the SNM, considering 150 and 200 scenarios. Note that in the case of the SNM $j$ can not be 0.

\begin{table}[H]
\frenchspacing
    \centering
    
\setlength\extrarowheight{2pt}
\resizebox{0.6\textwidth}{!}{%
    \begin{tabular}{P{25mm}| P{10mm}|P{13mm} P{13mm}|P{13mm} P{13mm}}
     \hline
        \multirow{2}{*}{S\textbf{cenarios}} & \multirow{2}{*}{\textbf{j}} & \multicolumn{2}{ c| }{\textbf{Time avg. (s)}}   & \multicolumn{2}{ c }{ \textbf{Iter. Avg.}}   \\ 
        ~ & ~ & \textbf{FPA} & \textbf{SNM} & \textbf{FPA} & \textbf{SNM} \\ \hline
        \multirow{6}{*}{50} & 0 & 16.86 & - & 2531 & - \\ 
        ~ & 4 & 16.99 & 2.81 & 2540 & 14 \\ 
        ~ & 10 & 15.30 & 3.90 & 2553 & 35 \\ 
        ~ & 20 & 17.62 & 4.35 & 2576 & 53 \\
        ~ & 30 & 22.10 & 5.75 & 2599 & 76 \\ 
        ~ & 50 & 19.72 & 9.56 & 2644 & 123 \\ \hline
        \multirow{6}{*}{150}& 4 & 143.19 & - & 5691 & - \\ 
        ~ & 10 & 114.25 & - & 5703 & - \\ 
        ~ & 20 & 90.01 & 14.20 & 5723 & 51 \\ 
        ~ & 30 & 92.65 & 21.08 & 5743 & 75 \\ 
        ~ & 40 & 92.20 & 28.87 & 5763 & 98 \\ 
        ~ & 50 & 91.70 & 29.40 & 5783 & 121 \\ \hline
        
    \end{tabular}}
    \captionsetup{width=.7\textwidth}
     \caption{Average time and iterations for each method when the parameter $r$ increases, considering $m_1=m_2=10.$ }
        \label{table:param_r}
    
\end{table}
 Observe that in general $\|M_\xi\|$ and therefore the parameter $r$ for the Fixed-point algorithm is quite significant considering the worked dimensions. For the problems solved here this $r$ is of order $10^3$, while for the SNM, we set $r\leq 50$. As expected, the performance of the IPHA is considerably better when applying the SNM. However, for small choices of the parameter $r$, we observed some instability in the resolution of the subproblems via SNM, generating the same behavior in the IPHA.

In almost all cases, we see a vast difference between solving the subproblems with one or another method. One possible reason is that solving the subproblems with the SNM gives a better approximation of the subproblem solution in fewer iterations; perhaps just a single iteration of SNM obtains a considerably small error than FPA.

\section{Conclusions and future work}
We presented the Progressive Hedging Algorithm for Stochastic Variational Inequalities where the generated subproblems can be approximately solved with an implementable tolerance condition. The scheme generalizes the exact version of the Progressive Hedging Algorithm, providing stronger convergence results.  We showed some numerical experiments in two-stage Nash games, solving the subproblems with the fixed point algorithm, and a semismooth Newton method, observing the advantages of each.

The main disadvantage of this scheme lies in the invariability of the parameter $r$. If we allowed its choice in each iteration, even stronger convergence results could be obtained. However, said generalization does not seem to be obtainable by proceeding as in this work. An alternative approach is introducing metric variables in the form of \cite{pls:08} that allow $r$ to be embedded in the respective matrices.




 \bibliographystyle{elsarticle-num} 
 \bibliography{IPHA}





\end{document}